\documentclass[letterpaper, 10 pt, conference]{ieeeconf}  

\IEEEoverridecommandlockouts                              
                                                          
\usepackage{amsmath}
\usepackage{amssymb}
\usepackage{mathtools}

\newtheorem{thm}{Theorem}

\newtheorem{prop}{Proposition}
\newtheorem{defn}{Definition}
\newtheorem{rem}{Remark}
\newtheorem{exmp}{Example}

\title{\LARGE \bf On the Exact Linearization of Multi-Input Discrete-Time Flat Systems}

\author{Johannes Schrotshamer, Bernd Kolar, Markus Sch{\"o}berl
\thanks{This research was funded in whole, or in part, by the Austrian Science Fund (FWF) P36473. For the purpose of open access, the author has applied a CC BY public copyright license to any Author Accepted Manuscript version arising from this submission. All authors are with the Institute of Control Systems, Johannes Kepler University Linz, Altenbergerstrasse 69, 4040 Linz, Austria.\newline {Email: \tt\small johannes.schrotshamer@jku.at, \newline bernd.kolar@jku.at, markus.schoeberl@jku.at}}}

\begin{document}

\maketitle
\thispagestyle{empty}
\pagestyle{empty}

\begin{abstract}                
In this contribution, we address the exact linearization of forward- and backward-flat discrete-time nonlinear systems with an arbitrary number of inputs. Building on structural properties of the flat parameterization, we show that forward-flat systems can be rendered static feedback linearizable by prolongations of suitably transformed inputs. Analogously, for backward-flat systems, we derive a constructive procedure showing that an exact linearization can be achieved by prelongations of suitably chosen functions of the system variables. In this way, we extend previously established results for two-input systems to the general \(m\)-input case. Finally, we illustrate the proposed linearization procedures by an academic example and a discrete-time model of a quadrotor.
\end{abstract}



\section{Introduction}
    The concept of differential flatness was introduced in the 1990s by Fliess, L{\'e}vine, Martin and Rouchon for nonlinear continuous-time systems (see e.g. \cite{FliessLevineMartinRouchon:1995,FliessLevineMartinRouchon:1999}) and has since become an important tool in nonlinear control theory. A continuous-time system is called flat if all system variables can be parameterized by a so-called flat output and a finite number of its time derivatives, while the flat output itself is a function of the system variables and a finite number of their time derivatives. Consequently, the solutions of a flat system are in one-to-one correspondence with those of a trivial system, which considerably simplifies both feedforward and feedback control design.
    For this reason, flatness-based methods for trajectory planning and tracking control have been successfully applied to a wide range of practical problems.
    However, no verifiable necessary and sufficient conditions for the flatness of general nonlinear multi-input systems are known to date; existing results are restricted to certain subclasses (see, e.g., \cite{SchoeberlRiegerSchlacher:2010,Schoeberl:2014,NicolauRespondek:2017,GstottnerKolarSchoberl:2021b,GstoettnerKolarSchoeberl:2022b,Levine:2025,NicolauRespondekLi:2025}).

    For the discrete-time analogue of differential flatness, known as difference flatness, several definitions have been proposed in the literature over the years. To the best of our knowledge, the most general definition of discrete-time flatness can be found in \cite{DiwoldKolarSchoberl:2020}, which is based on the one-to-one correspondence between the solutions of a flat system and those of a trivial system. Within this approach, both the flat output and the flat parameterization may depend on forward- as well as backward-shifts (see also \cite{GuillotMillerioux2020} or \cite{Kaldmae:2021}).
    For this general class of discrete-time systems, some promising initial results concerning verifiable necessary and sufficient conditions were obtained in \cite{Kaldmae:2021} using an algebraic approach. However, due to the associated computational complexity, arising from the presence of degrees of freedom and the need to solve partial differential equations, the development of alternative methods remains an active field of research.

    As discussed in \cite{SchrotshamerKolarSchoeberl:2025}, this general notion of difference flatness contains forward-flat and backward-flat systems as important subclasses. Forward-flat systems, which have been studied, e.g., in \cite{Sira-RamirezAgrawal:2004,KaldmaeKotta:2013,KolarKaldmaeSchoberlKottaSchlacher:2016}, are characterized by the fact that both the flat output and the flat parameterization of the system variables depend exclusively on forward-shifts. For this subclass, two distinct geometric tests providing verifiable necessary and sufficient conditions have been developed in \cite{KolarDiwoldSchoberl:2023,KolarSchrotshamerSchoberl:2025}.
    Backward-flat systems, on the other hand, arise when the flat output and the flat parameterization are restricted to backward-shifts. For this subclass, an algorithmic test was recently introduced in \cite{SchrotshamerKolarSchoeberl:2026}.

    It has been shown in \cite{DiwoldKolarSchoberl:2020} that, analogously to the continuous-time case, every flat discrete-time system can be exactly linearized by a suitable dynamic feedback. Additionally, in \cite{SchrotshamerKolarSchoeberl:2025} it was proven that the properties of forward- and backward-flatness are related to rank conditions regarding certain submatrices of the Jacobian of the flat parameterization. In particular, for forward-flat systems, the submatrix of the Jacobian corresponding to the input parameterization with respect to the highest occurring forward-shifts is not of full rank. Conversely, in the backward-flat case, the analogous submatrix associated with the state parameterization and the highest occurring backward-shifts fails to have full rank.
    However, their particular location within the Jacobian suggests that these rank defects can be resolved by particularly simple types of dynamic feedback. As established in \cite{SchrotshamerKolarSchoeberl:2025} for two-input systems, these consist of shift chains, which are attached either in the form of forward-shifts to suitably chosen transformed inputs or as backward-shifts to suitably chosen functions of the system variables. These two types of dynamic feedback are referred to as prolongations and prelongations, respectively.

    In this contribution, we extend the results of \cite{SchrotshamerKolarSchoeberl:2025} from the two-input case to forward-flat and backward-flat systems with $m$ inputs. In particular, we show that forward-flat systems can be rendered static feedback linearizable by prolongations of suitably transformed inputs, whereas backward-flat systems can be rendered static feedback linearizable by prelongations of suitably chosen functions of the system variables.

    The paper is organized as follows: In Section \ref{sec:two}, we first recall the concept of flatness for discrete-time systems as discussed in \cite{DiwoldKolarSchoberl:2020} including the subclasses of forward- and backward-flat systems. Next, in Section \ref{sec:three}, for the special cases of forward- and backward-flat systems with $m$ inputs, we derive simple types of dynamic extension for the purpose of an exact linearization before we illustrate our results by an academic example and the practical example of an unmanned aerial vehicle in the form of a quadrotor in Section \ref{sec:four}.

    \paragraph*{Notation}
    To enhance readability, we commonly use the superscript $+$ to indicate a one-step forward-shift of a variable. Higher-order forward-shifts and backward-shifts are, in contrast, written using bracketed subscripts. Thus, for $\alpha\geq0$, $u_{[\alpha]}$ denotes the $\alpha$th forward-shift of the variable $u$, whereas $u_{[-\alpha]}$ denotes its $\alpha$th backward-shift. The corresponding forward- and backward-shift operator denoted by $\delta^\alpha$ and $\delta^{-\alpha}$ can act on variables as well as functions, see e.g. \cite{DiwoldKolarSchoberl:2020} for details.
    When a variable consists of several components, i.e., $y=(y^1,\ldots,y^m)$, we write $y_{[\alpha]}=(y^1_{[\alpha]},\ldots,y^m_{[\alpha]})$.
    Shifts of different orders are represented compactly using multi-indices. Let $S_1=(s_{1,1},\ldots,s_{1,m})$ and $S_2=(s_{2,1},\ldots,s_{2,m})$ be two multi-indices. Then, $y_{[-S_1]}$ and $y_{[S_2]}$ denote the corresponding backward- and forward-shifts of the individual components of $y$, respectively. Furthermore, $y_{[-S_1,S_2]}$ is an abbreviation for $y$ together with its backward-shifts up to the orders $S_1$ and its forward-shifts up to the orders $S_2$.
    
    If $s_{1,j}\leq s_{2,j}$ holds for all $j=1,\ldots,m$, we write $S_1\leq S_2$. In addition, $y^j_{[s_{1,j},s_{2,j}]}$ is defined to be empty whenever $s_{1,j}>s_{2,j}$. Arithmetic operations on multi-indices are understood componentwise. In particular,
    \[
        S_1\pm S_2
        =
        (s_{1,1}\pm s_{2,1},\ldots,s_{1,m}\pm s_{2,m})\,.
    \]
    Similarly, addition and subtraction of a scalar $a$ are defined by
    \[
        S_1\pm a
        =
        (s_{1,1}\pm a,\ldots,s_{1,m}\pm a).
    \]
    Finally,
    \[
        \#S_1=\sum_{j=1}^{m}s_{1,j}
    \]
    denotes the sum of the components of $S_1$.

\section{Flatness of discrete-time systems} \label{sec:two}
    In the following, we consider nonlinear time-invariant discrete-time systems of the form
    \begin{equation}\label{eq:sysEq}
      x^{i,+} = f^{i}(x,u), \quad i=1,\dots,n,
    \end{equation}
    with $\dim(x)=n$, $\dim(u)=m$, and smooth functions $f^{i}(x,u)$.
    We further assume that the submersivity condition $\operatorname{rank}(\partial_{(x,u)}f)=n$ holds for the system equations \eqref{eq:sysEq}.
    Since submersivity is necessary for accessibility and thus also for flatness, it is a common assumption in the discrete-time literature. Under this assumption, there always exist $m$ functions $g(x, u)$ such that the map
    \begin{equation}\label{eq:sysEq_ext}
		x^{+} = f(x, u)\,,\quad \zeta = g(x, u)
	\end{equation}
	is locally a diffeomorphism with its inverse given by
    \begin{equation}\label{eq:sysEq_ext_inv}
		x=\psi_x(x^+,\zeta)\,, \quad u=\psi_u(x^+,\zeta)\,.
	\end{equation}
    As defined in \cite{DiwoldKolarSchoberl:2020}, a discrete-time system \eqref{eq:sysEq} is called flat if there exists a one-to-one correspondence between its solutions $(x(k),u(k))$ and the solutions $y(k)$ of a trivial system with the same number of inputs. According to Definition~1 of \cite{DiwoldKolarSchoberl:2020}, this correspondence implies that the values of $x(k)$ and $u(k)$ at a given time step may depend on an arbitrary but finite number of past and future values of $y(k)$, whereas $y(k)$ may depend on an arbitrary but finite number of past and future values of $x(k)$ and $u(k)$.\\
    Next, we recall the modified yet equivalent formulation of flatness used in \cite{SchrotshamerKolarSchoeberl:2025}, together with the subclasses of forward- and backward-flat systems. These arise by restricting the flat output and flat parameterization to forward- and backward-shifts, respectively.
    \vspace{0.05cm}

	\begin{defn}
		\label{def:flatness}The system \eqref{eq:sysEq} is said to be flat around an equilibrium $(x_{0},u_{0})$, if the $n+m$ coordinate functions $x$ and $u$ can be expressed locally by an $m$-tuple of functions
		\begin{equation}
			y^{j}=\varphi^{j}(\zeta_{[-Q_1,-1]},x,u,u_{[1,Q_2]})\,,\quad\,j=1,\ldots,m\label{eq:flat_output}
		\end{equation}
		called the flat output, and its backward- and forward-shifts up to some finite order, i.e.
    	\begin{equation}\label{eq:flat_param}
    		\begin{aligned}
    			x&=F_{x}(y_{[-R_1,R_2-1]})\,, &\; u&=F_{u}(y_{[-R_1,R_2]})\,.
    		\end{aligned}
    	\end{equation}
        \vspace{-0.3cm}
	\end{defn}
    For a given flat output \eqref{eq:flat_output}, the parameterization \eqref{eq:flat_param} of the state- and input variables $x$ and $u$ is unique.\\ 
    The special case of forward-flatness is obtained from Definition~\ref{def:flatness} if no backward-shifts occur in the flat output \eqref{eq:flat_output} and $R_1=(0,\ldots,0)$ holds for the parameterization \eqref{eq:flat_param}.
    \vspace{0.05cm}
	\begin{defn}
		\label{def:forward-flatness}The system \eqref{eq:sysEq} is said
		to be forward-flat around an equilibrium $(x_{0},u_{0})$, if the
		$n+m$ coordinate functions $x$ and $u$ can be expressed locally
		by an $m$-tuple of functions
		\begin{equation*}
			y^{j}=\varphi^{j}(x,u,u_{[1,Q_2]})\,,\quad j=1,\ldots,m\label{eq:forward-flat_output}
		\end{equation*}
		called the forward-flat output, and its forward-shifts up to some finite order, i.e.
    	\begin{equation*}
    		\begin{aligned}
    			x&=F_{x}(y_{[0,R_2-1]})\,, &\;u&=F_{u}(y_{[0,R_2]})\,.
    		\end{aligned}
    	\end{equation*}
        \vspace{-0.3cm}
	\end{defn}
    Conversely, if the flat output \eqref{eq:flat_output} does not involve any forward-shifts and $R_2=(0,\ldots,0)$ holds for the parameterization \eqref{eq:flat_param}, then Definition \ref{def:flatness} leads to the subclass of backward-flat systems.
    \vspace{0.05cm}
	\begin{defn}
		\label{def:backward-flatness}The system \eqref{eq:sysEq} is said
		to be backward-flat around an equilibrium $(x_{0},u_{0})$, if the
		$n+m$ coordinate functions $x$ and $u$ can be expressed locally
		by an $m$-tuple of functions
		\begin{equation}
			y^{j}=\varphi^{j}(\zeta_{[-Q_1,-1]},x,u)\,,\quad j=1,\ldots,m\label{eq:backward-flat_output}
		\end{equation}
		called the backward-flat output, and its backward-shifts up to some finite order, i.e.
        \begin{equation}\label{eq:backward_flat_param}
    		\begin{aligned}
    			x&=F_{x}(y_{[-R_1,-1]})\,, &\; u&=F_{u}(y_{[-R_1,0]})\,.
    		\end{aligned}
	    \end{equation}
        \vspace{-0.3cm}
	\end{defn}

    For systems that are flat in the sense of Definition \ref{def:flatness}, Definition \ref{def:forward-flatness} or Definition \ref{def:backward-flatness}, with $R=R_1+R_2$, the inequality $\# R\geq n$ is satisfied, and the parameterizing map $\left(F_x, F_u\right): \mathbb{R}^{\# R+m} \rightarrow \mathbb{R}^{n+m}$ is a submersion. Let $d$ denote the difference between the dimensions of the domain and the co-domain of this map, that is, $d=\# R+m-(n+m)=\# R-n$. For a flat output $y$, this difference $d$ equals the minimal dimension of the state of a linearizing dynamic feedback. In the special case $\# R=n$, the parameterizing map $\left(F_x, F_u\right): \mathbb{R}^{n+m} \rightarrow \mathbb{R}^{n+m}$ becomes a diffeomorphism, and the system can be exactly linearized by static feedback. Flat outputs with this property are called linearizing outputs.

\section{Exact linearization} \label{sec:three}
    The exact linearization of discrete-time flat systems has already been addressed in \cite{DiwoldKolarSchoberl:2020}, where it was shown that for every discrete-time system that is flat in the sense of Definition \ref{def:flatness}, there exists a linearizing endogenous dynamic feedback. In this section, we show that the subclasses of forward-flat and backward-flat systems with $m$ inputs, to which the result of \cite{DiwoldKolarSchoberl:2020} also applies, can be rendered static feedback linearizable by special types of dynamic feedback.

    For this purpose, we recall from \cite{SchrotshamerKolarSchoeberl:2025} that the properties of forward- and backward-flatness are related to rank conditions regarding certain submatrices of the Jacobian of the flat parameterization. Specifically, depending on whether a given system is forward- or backward-flat, rank defects occur either in the submatrix $\partial_{y_{[R_2]}}F_u$ or in $\partial_{y_{[-R_1]}}F_x$. However, due to their particular location within the Jacobian, these rank defects can be resolved by particularly simple types of dynamic feedback. As in the two-input case considered in \cite{SchrotshamerKolarSchoeberl:2025}, these consist of shift chains, which are attached either in the form of forward-shifts to suitably chosen transformed inputs or as backward-shifts to suitably chosen functions of the system variables. These two types of dynamic feedback are referred to as prolongations and prelongations, respectively.

    \subsection{Linearization of Multi-Input Forward-Flat Systems}
        In the following, we prove that forward-flat systems, which according to Theorem~12 in \cite{KolarDiwoldSchoberl:2021} admit a flat output of the form $y=\varphi(x)$, can be rendered static feedback linearizable through prolongations of suitably chosen transformed inputs. To this end, we use results from \cite{KolarDiwoldGstottnerSchoeberl:2022} derived for flat systems in the general sense of Definition~\ref{def:flatness}. The main idea of the procedure proposed therein is to transform the forward-shifts of the flat output into a special form by successively replacing the original inputs $u$ with new coordinates $\bar u$. This stepwise construction, which provides the basis for the subsequent design of a tracking controller, is carried out recursively over a total of $s$ steps and carries over directly to the special case of forward-flat systems.

        Since our results build on the procedure proposed in \cite{KolarDiwoldGstottnerSchoeberl:2022}, we briefly recall its essential steps. Initially, the components of the forward-flat output $y$ are shifted forward until they depend explicitly on the coordinates $u$, and as many of these functions as possible are introduced as new coordinates ${\bar u}_1$ with $\operatorname{dim}({\bar u}_1)=m_1$. In each subsequent step $i=2,\ldots,s$, the rest of the output components are shifted further until they depend explicitly on the remaining components $u_{\text{rest}_{i-1}}$, and a maximal set of $m_i$ such functions is introduced as new coordinates ${\bar u}_i$. Repeating this procedure until all $m=m_1+\ldots+m_s$ components of $u$ have been replaced yields a minimal multi-index $\kappa=(\kappa_1,\ldots,\kappa_s)$ such that ${\bar u}_i=\delta^{\kappa_i}(\varphi_i)$ with $i=1,\ldots,s$.
        Due to this stepwise construction where in every step $m_i$ new coordinates are introduced, it follows that $y=(y_1,\ldots,y_s)$, $\varphi=(\varphi_1,\ldots,\varphi_s)$, ${\bar u}=({\bar u}_1,\ldots,{\bar u}_s)$, $R_2=(r_{2,1},\ldots,r_{2,s})$ with $y_{i}=(y_i^1,\ldots,y_i^{m_i})$, $\varphi_{i}=(\varphi_i^1,\ldots,\varphi_i^{m_i})$, ${\bar u}_i=({\bar u}_i^1,\ldots,{\bar u}_i^{m_i})$, $r_{2,i}=(r_{2,i}^1,\ldots,r_{2,i}^{m_i})$ and $\kappa_{i}=(\kappa_{i}^1,\ldots,\kappa_{i}^{m_i})$.

        This stepwise construction gives rise to the invertible coordinate transformations
        \begin{equation}\label{eq:trafo_v_prolongation}
            \begin{aligned}
                {\bar u}_1 &= \varphi_{1,[\kappa_1]}(x,u) \\
                {\bar u}_2 &= \varphi_{2,[\kappa_2]}(x,{\bar u}_1,{\bar u}_{1,[1]},\ldots,u_{\text{rest}_1}) \\
                & \hspace{0.2cm} \vdots \\
                {\bar u}_s &= \varphi_{s,[\kappa_s]}(x,{\bar u}_1,{\bar u}_{1,[1]},\ldots,{\bar u}_{s-1},{\bar u}_{s-1,[1]},\ldots,u_{\text{rest}_{s-1}})\,,
            \end{aligned}
        \end{equation}
        which can be abbreviated as ${\bar u}=\Phi_u(x,{\bar u}_{[0,R_2-\kappa]},u)$. Owing to the triangular structure of \eqref{eq:trafo_v_prolongation}, the transformed inputs and their forward-shifts appearing on the right-hand side can be successively substituted, so that the transformation can ultimately be expressed solely in terms of $x$, $u$, and forward-shifts of $u$. 
        Since the transformation involves forward-shifts of the inputs, it represents a generalized input transformation.
        
        Applying the transformation \eqref{eq:trafo_v_prolongation} together with the corresponding transformations of its forward-shifts
        \begin{equation*}\label{eq:trafo_v_plus_prolongation}
            \begin{aligned}
                {\bar u}_{i,[1]} &= \varphi_{i,[\kappa_{i}+1]}(x,{\bar u}_1,{\bar u}_{1,[1]},\ldots,{\bar u}_{i-1},{\bar u}_{i-1,[1]},\ldots,\\
                & \hspace{1.8cm} u_{\text{rest}_{i-1}},u_{\text{rest}_{i-1,[1]}}) \\
                {\bar u}_{i,[2]} &= \varphi_{i,[\kappa_{i}+2]}(x,{\bar u}_1,{\bar u}_{1,[1]},\ldots,{\bar u}_{i-1},{\bar u}_{i-1,[1]},\ldots,\\
                & \hspace{1.8cm} u_{\text{rest}_{i-1}},u_{\text{rest}_{i-1,[1]}},u_{\text{rest}_{i-1,[2]}}) \\
                & \hspace{0.2cm} \vdots 
            \end{aligned}
        \end{equation*}
        for $i=1,\ldots,s-1$, transforms the forward-shifts of the flat output $y$ up to the orders $R_2$ into the form
        \begin{equation}\label{eq:proof_prop_prolonged}
            \begin{aligned}
                y_{1,[0, \kappa_1-1]} & =\varphi_{1,[0, \kappa_1-1]}(x) \\
                y_{1,[\kappa_1, r_{2,1}]} & ={\bar u}_{1,[0, r_{2,1}-\kappa_1]} \\
                & \hspace{0.2cm} \vdots \\
                y_{s-1,[0, \kappa_{s-1}-1]} & =\varphi_{s-1,[0, \kappa_{s-1}-1]}(x, {\bar u}_{1,[0, r_{2,1}-\kappa_1-1]}, \\
                    & \hspace{0.4cm} \ldots, {\bar u}_{s-2,[0, r_{2,s-2}-\kappa_{s-2}-1]}) \\
                y_{s-1,[\kappa_{s-1}, r_{2,s-1}]} & ={\bar u}_{{s-1},[0, r_{2,s-1}-\kappa_{s-1}]}\\
                y_{s,[0, r_{2,s}-1]} & =\varphi_{s,[0, r_{2,s}-1]}(x, {\bar u}_{1,[0, r_{2,1}-\kappa_1-1]}, \\
                &\hspace{0.4cm} \ldots, {\bar u}_{s-1,[0, r_{2,s-1}-\kappa_{s-1}-1]}) \\
                y_{s,[r_{2,s}]} & ={\bar u}_{s}\,.
            \end{aligned}
        \end{equation}
        
        With these preliminaries, we can state the following theorem.
        
        \begin{thm}\label{thm:prolongation}
            For every forward-flat system \eqref{eq:sysEq} with a forward-flat output of the form $y=\varphi(x)$, there exists an invertible, generalized input transformation $u=\hat{\Phi}_u(x,{\bar u}_{[0,R_2-\kappa]})$, such that the $d_2$-fold ($d_2=\# R_2-\#\kappa$) prolonged system
    		\begin{equation}\label{eq:prolonged_sys}
    			\begin{aligned}
    				\Sigma^{(d_2)}: x^{i,+} & =f^i(x, \hat{\Phi}_u(x,{\bar u}_{[0,R_2-\kappa]})) \\
    				{\bar u}^{+}_1 & ={\bar u}_{1,[1]} \\
    				& \hspace{0.2cm}\vdots \\
    				{\bar u}^{+}_{1,[r_{2,1}-\kappa_1-1]} & ={\bar u}_{1,[r_{2,1}-\kappa_1]}\\
                    &\hspace{0.2cm} \vdots \\
                    {\bar u}^{+}_{s-1} & ={\bar u}_{s-1,[1]} \\
    				&\hspace{0.2cm} \vdots \\
    				{\bar u}^{+}_{s-1,[r_{2,s-1}-\kappa_{s-1}-1]} & ={\bar u}_{s-1,[r_{2,s-1}-\kappa_{s-1}]}
    			\end{aligned}
    		\end{equation}
    		with the state $(x^1, \ldots, x^n, {\bar u}_{1,[0,r_{2,1}-\kappa_1-1]}, \ldots,$ ${\bar u}_{s-1,[0,r_{2,s-1}-\kappa_{s-1}-1]})$ and the input $({\bar u}_{1,[r_{2,1}-\kappa_1]},\ldots,$ ${\bar u}_{s-1,[r_{2,s-1}-\kappa_{s-1}]},{\bar u}_s)$ is exactly linearizable by static feedback.
    	\end{thm}
    
        \begin{proof}
            The key step of the proof is to show that the map \eqref{eq:proof_prop_prolonged} is a diffeomorphism. Since the forward-shifts of a flat output up to an arbitrary order are functionally independent, it suffices to show that the number of variables $x,\bar{u}_{[0,R_2-\kappa]}$ on the right-hand side of \eqref{eq:proof_prop_prolonged} equals the number of forward-shifts $y_{[0,R_2]}$ on the left-hand side, i.e.,
            \begin{equation*}\label{eq:dim_argument_forward}
                n+\#(R_2-\kappa)+m=\#R_2+m.
            \end{equation*}
            This equality follows immediately from $\#\kappa=n$, which, according to Theorem~4.5 in \cite{KolarDiwoldGstottnerSchoeberl:2022}, holds for every forward-flat system. Moreover, the map \eqref{eq:proof_prop_prolonged} is precisely the inverse of the parameterizing map of the prolonged system \eqref{eq:prolonged_sys}. Consequently, the parameterizing map is itself a diffeomorphism, and the prolonged system is therefore static feedback linearizable, which completes the proof.
        \end{proof}

        \begin{rem}
            Having shown that the prolonged system \eqref{eq:prolonged_sys} is static feedback linearizable, the map \eqref{eq:proof_prop_prolonged} can be used to transform it into Brunovsky normal form.
        \end{rem}

        \begin{rem}
            As already mentioned, the particular structure of $y$ and its forward-shifts in \eqref{eq:proof_prop_prolonged} represents the specialization to forward-flat systems of the more general form derived in \cite{KolarDiwoldGstottnerSchoeberl:2022} for flat systems that may also involve backward-shifts.
            In that work, the main objective was to derive a quasistatic-like feedback for the purpose of trajectory tracking.
            In the present contribution, however, we focus on dynamic feedback linearization and show that the required dynamic extension is of a particularly simple type, namely prolongations of suitably transformed inputs.
        \end{rem}

        A procedure analogous to the one in \cite{KolarDiwoldGstottnerSchoeberl:2022} can also be derived for the backward-flatness case. Based on this procedure, we show in the following that a linearizing dynamic feedback for backward-flat systems can be constructed using similarly simple dynamic extensions.

    \subsection{Linearization of Multi-Input Backward-Flat Systems}
        We now turn to the subclass of backward-flat systems in the sense of Definition~\ref{def:backward-flatness}.
        Although backward-flat outputs are in general of the form \eqref{eq:backward-flat_output}, as shown in \cite{SchrotshamerKolarSchoeberl:2025,SchrotshamerKolarSchoeberl:2026}, the rank condition $\operatorname{rank}(\partial_u \varphi_{[0]}) = m$\footnote{The rank condition $\operatorname{rank}(\partial_u \varphi_{[0]}) = m$ must hold because substituting the map \eqref{eq:sysEq_ext_inv} into the backward-shifts of \eqref{eq:y_backward} yields functions of $x$ and $\zeta_{[-1]},\zeta_{[-2]},\ldots$ only. Otherwise, not all components of $u$ could be expressed in terms of $y_{[-R_1,0]}$.} must hold and every backward-flat system admits a flat output of the form
        \begin{equation}\label{eq:y_backward}
        y^j=\varphi^j(x,u),
        \qquad j=1,\ldots,m\,.
        \end{equation}
        For backward-flat systems with a flat output of the form \eqref{eq:y_backward}, we show that a linearizing dynamic feedback can be constructed from prelongations of suitably chosen functions of the system variables.
        As a first step, we establish that after a suitable coordinate transformation, the backward-shifts $y_{[-R_1,0]}$ of the flat output possess a special structure. For this purpose, we adapt the method proposed in \cite{KolarDiwoldGstottnerSchoeberl:2022} for backward-flat systems. The main idea is to successively replace the coordinates $\ldots,\zeta_{[-2]},\zeta_{[-1]}$ appearing in $\varphi_{[-R_1,0]}$ by the backward-shifts $\ldots,\bar{\zeta}_{[-2]},\bar{\zeta}_{[-1]}$ of the backward-flat output that leads to the construction of a multi-index $\lambda$ such that $\bar\zeta_{[-1-\alpha]}=\delta^{-\lambda-\alpha}(\varphi)$ for $\alpha\geq 0$. Analogously to the procedure in \cite{KolarDiwoldGstottnerSchoeberl:2022}, the required transformation is constructed successively over $s$ steps, indexed by $i=1,\ldots,s$.\\

        In the first step, the components of the backward-flat output \eqref{eq:y_backward} are shifted backward until they depend explicitly on at least one of the coordinates $\zeta_{[-1]}$, i.e., the multi-index $L_1=(l_1^1, \ldots, l_1^m)$ is determined such that
        \begin{equation}\label{eq:backward_L_1}
            \begin{aligned}
            \delta^{-l_1^j}(\varphi^j) & =\varphi_{[-l_1^j]}^j(\zeta_{[-1]}, x) \\
            \delta^{-l_1^j+1}(\varphi^j) & =\varphi_{[-l_1^j+1]}^j(x,u)\,.
            \end{aligned}
        \end{equation}
        Note that the variables $u$ do not necessarily occur in \eqref{eq:backward_L_1}.
        Next, we determine $m_1=\operatorname{rank}(\partial_{\zeta_{[-1]}} \varphi_{[-L_1]})$ and reorder the components of the flat output \eqref{eq:y_backward} and $\zeta$ such that $\operatorname{rank}(\partial_{\zeta_{1,[-1]}} \varphi_{1,[-\lambda_1]})=m_1$, where $\varphi_1=(\varphi^1, \ldots, \varphi^{m_1})$, $\zeta_{1}=(\zeta^1, \ldots, \zeta^{m_1})$, and $\lambda_1=(l_1^1, \ldots, l_1^{m_1})$ consist of the first $m_1$ components of $\varphi, \zeta$, and $L_1$, respectively.
        Subsequently, the $m_1$ functions $\delta^{-\lambda_1}(\varphi_{1})$ are introduced as new coordinates $\bar\zeta_{1,[-1]}$ by means of the coordinate transformation
        \begin{equation}\label{eq:trafo_zeta1_bar}
            \begin{aligned}
            \bar\zeta_{1,[-1]} & =\varphi_{1,[-\lambda_1]}(\zeta_{[-1]}, x) \\
            \zeta_{\text{rest}_1,[-1]} & = (\zeta_{[-1]}^{m_1+1}, \ldots, \zeta_{[-1]}^m) \\
            \bar\zeta_{1,[-2]} & =\varphi_{1,[-\lambda_1-1]}(\zeta_{[-2,-1]}, x) \\
            \zeta_{\text{rest}_1,[-2]} & =(\zeta_{[-2]}^{m_1+1}, \ldots, \zeta_{[-2]}^m) \\
            \bar\zeta_{1,[-3]} & =\varphi_{1,[-\lambda_1-2]}(\zeta_{[-3,-1]}, x) \\
            \zeta_{\text{rest}_1,[-3]} & =(\zeta_{[-3]}^{m_1+1}, \ldots, \zeta_{[-3]}^m)\,, \\
            & \hspace{0.2cm}\vdots
            \end{aligned}
        \end{equation}
        replacing $\zeta_{1,[-1]}$ and their backward-shifts by $\bar\zeta_{1,[-1]}$ and their backward-shifts,
        while leaving the remaining coordinates $\zeta_{\text{rest}_1,[-1]}=(\zeta_{[-1]}^{m_1+1},\ldots,\zeta_{[-1]}^m)$ and their backward-shifts unchanged.\\
        Under the coordinate transformation \eqref{eq:trafo_zeta1_bar}, the backward-shifts of $y$ take the form
        \begin{equation*}
            \begin{aligned}
            y_{1,[-\lambda_1]} &= \bar\zeta_{1,[-1]} \\
            y_{1,[-\lambda_1+1, 0]} &= \varphi_{1,[-\lambda_1+1, 0]}(x,u) \\
            y_{\text{rest}_1,[-L_{\text{rest}_1}]} &= \varphi_{\text{rest}_1,[-L_{\text{rest}_1}]}(\bar\zeta_{1,[-1]}, x) \\
            y_{\text{rest}_1,[-L_{\text{rest}_1}+1,0]} &= \varphi_{\text{rest}_1,[-L_{\text{rest}_1}+1,0]}(x,u)
            \end{aligned}
        \end{equation*}
        with $y_1=(y^1, \ldots, y^{m_1})$, $y_{\text{rest}_1}=(y^{m_1+1}, \ldots, y^m)$, $L_{\text{rest}_1}=(l_1^{m_1+1}, \ldots, l_1^m)$, and $\varphi_{\text{rest}_1,[\alpha]}=(\varphi_{[\alpha]}^{m_1+1}, \ldots, \varphi_{[\alpha]}^m) \circ \hat{\Phi}_1$ and $\hat{\Phi}_1$ denoting the inverse of the transformation \eqref{eq:trafo_zeta1_bar}. By the definition of $m_1$ as the rank of $\partial_{\zeta_{[-1]}}\varphi_{[-L_1]}$, the functions $\varphi_{\text{rest}_1,[-L_{\text{rest}_1}]}$ must be independent of the remaining coordinates $\zeta_{\text{rest}_1,[-1]}$; otherwise, the rank of $\partial_{\zeta_{[-1]}}\varphi_{[-L_1]}$ would exceed $m_1$.\\
    
        In the subsequent steps $i=2,\ldots,s-1$, the remaining components $y_{\text{rest}_{i-1}}$ of the flat output are shifted further until they depend explicitly on the remaining components $\zeta_{\text{rest}_{i-1,[-1]}}$ of the original variables $\zeta_{[-1]}$. That is, the multi-index $L_i=(l_i^1, \ldots, l_i^{m-(m_1+\ldots+m_{i-1})})$ is determined such that
        \begin{equation*}
            \begin{aligned}
            \delta^{-l_i^j}(\varphi_{\text{rest}_{i-1}}^j) &= \varphi_{\text{rest}_{i-1},[-l_i^j]}^j(\zeta_{\text{rest}_{i-1},[-1]}, \ldots, \bar\zeta_{{i-1},[-2]}, \\
            & \hspace{0.5cm} \bar\zeta_{{i-1},[-1]}, \ldots, \bar\zeta_{{1},[-2]}, \bar\zeta_{{1},[-1]}, x) \\
            \delta^{-l_i^j+1}(\varphi_{\text{rest}_{i-1}}^j) &= \varphi_{\text{rest}_{i-1},[-l_i^j+1]}^j(\ldots, \bar\zeta_{{i-1},[-2]}, \bar\zeta_{{i-1},[-1]}, \\
            & \hspace{0.5cm} \ldots, \bar\zeta_{{1},[-2]}, \bar\zeta_{{1},[-1]}, x, u)\,.
            \end{aligned}
        \end{equation*}
        Again, we determine $m_i=\operatorname{rank}(\partial_{\zeta_{\text{rest}_{i-1},[-1]}} \varphi_{\text{rest}_{i-1},[-L_i]})$ and reorder the components of the flat output belonging to $y_{\text{rest}_{i-1}}$ and the components of $\zeta_{\text{rest}_{i-1}}$ such that $\operatorname{rank}(\partial_{\zeta_{i,[-1]}} \varphi_{i,[-\lambda_i]})=m_i$, where $\varphi_i=(\varphi_{\text{rest}_{i-1}}^1, \ldots, \varphi_{\text{rest}_{i-1}}^{m_i})$, $\zeta_{i}=(\zeta_{\text{rest}_{i-1}}^1, \ldots, \zeta_{\text{rest}_{i-1}}^{m_i})$, and $\lambda_i=(l_i^1, \ldots, l_i^{m_i})$ consist of the first $m_i$ components of $\varphi_{\text{rest}_{i-1}}$, $\zeta_{\text{rest}_{i-1}}$, and $L_i$, respectively.
        Subsequently, the $m_i$ functions $\delta^{-\lambda_i}(\varphi_i)$ are introduced as new coordinates $\bar\zeta_{i,[-1]}$ by means of the coordinate transformation
        \begin{equation}\label{eq:trafo_zeta2_bar}
            \begin{aligned}
            \bar\zeta_{i,[-1]} & =\varphi_{i,[-\lambda_i]}(\zeta_{\text{rest}_{i-1},[-1]},\ldots, \bar\zeta_{i-1,[-2]}, \bar\zeta_{i-1,[-1]}, \ldots, \\
            & \hspace{0.5cm} \bar\zeta_{1,[-2]}, \bar\zeta_{1,[-1]}, x) \\
            \zeta_{\text{rest}_i,[-1]} & = (\zeta_{[-1]}^{m_1+\ldots+m_i+1}, \ldots, \zeta_{[-1]}^m) \\
            \bar\zeta_{i,[-2]} & =\varphi_{i,[-\lambda_i-1]}(\zeta_{\text{rest}_{i-1},[-2,-1]}, \ldots, \bar\zeta_{i-1,[-2]}, \bar\zeta_{i-1,[-1]}, \\
            & \hspace{0.5cm} \ldots, \bar\zeta_{1,[-2]}, \bar\zeta_{1,[-1]}, x) \\
            \zeta_{\text{rest}_i,[-2]} & =(\zeta_{[-2]}^{m_1+\ldots+m_i+1}, \ldots, \zeta_{[-2]}^m) \\
            \bar\zeta_{i,[-3]} & =\varphi_{i,[-\lambda_i-2]}(\zeta_{\text{rest}_{i-1},[-3,-1]}, \ldots, \bar\zeta_{i-1,[-2]}, \bar\zeta_{i-1,[-1]}, \\
            & \hspace{0.5cm} \ldots, \bar\zeta_{1,[-2]}, \bar\zeta_{1,[-1]}, x) \\
            \zeta_{\text{rest}_i,[-3]} & =(\zeta_{[-3]}^{m_1+\ldots+m_i+1}, \ldots, \zeta_{[-3]}^m)\,, \\
            & \hspace{0.2cm}\vdots
            \end{aligned}
        \end{equation}
        which replaces $\zeta_{i,[-1]}=(\zeta_{[-1]}^{m_1+\ldots+m_{i-1}+1}, \ldots, \zeta_{[-1]}^{m_1+\ldots+m_i})$ and their backward-shifts by $\bar\zeta_{i,[-1]}$ and their backward-shifts, while leaving the remaining coordinates $\zeta_{\text{rest}_i,[-1]}=(\zeta_{[-1]}^{m_1+\ldots+m_i+1}, \ldots, \zeta_{[-1]}^m)$ and their backward-shifts unchanged.\\
        Under the coordinate transformation \eqref{eq:trafo_zeta2_bar}, the backward-shifts of $y$ take the form
        \begin{equation*}
            \begin{aligned}
            y_{1,[-\lambda_1]} &= \bar\zeta_{1,[-1]} \\
            y_{1,[-\lambda_1+1, 0]} &= \varphi_{1,[-\lambda_1+1, 0]}(x,u) \\
            & \hspace{0.2cm} \vdots\\
            y_{i,[-\lambda_i]} &= \bar\zeta_{i,[-1]} \\
            y_{i,[-\lambda_i+1, 0]} &= \varphi_{i,[-\lambda_i+1, 0]}(\ldots,\bar\zeta_{i-1,[-2]},\bar\zeta_{i-1,[-1]},\ldots, \\
            & \hspace{0.5cm} \bar\zeta_{1,[-2]},\bar\zeta_{1,[-1]},x,u) \\
            y_{\text{rest}_i,[-L_{\text{rest}_i}]} &= \varphi_{\text{rest}_i,[-L_{\text{rest}_i}]}(\bar\zeta_{i,[-1]},\ldots,\bar\zeta_{i-1,[-2]}, \\
            & \hspace{0.5cm}\bar\zeta_{i-1,[-1]},\ldots,\bar\zeta_{1,[-2]},\bar\zeta_{1,[-1]},x) \\
            y_{\text{rest}_i,[-L_{\text{rest}_i}+1,0]} &= \varphi_{\text{rest}_i,[-L_{\text{rest}_i}+1,0]}(\ldots,\bar\zeta_{i-1,[-2]},\bar\zeta_{i-1,[-1]}, \\
            & \hspace{0.5cm} \ldots,\bar\zeta_{1,[-2]},\bar\zeta_{1,[-1]},x,u)
            \end{aligned}
        \end{equation*}    
        with $y_i=(y^{m_1+\ldots +m_{i-1}+1}, \ldots, y^{m_1+\ldots+m_i})$, $y_{\text{rest}_i}=(y^{m_1+\ldots+m_i+1}, \ldots, y^m)$, $L_{\text{rest}_i}=(l_i^{m_i+1}, \ldots, l_i^{m-(m_1+\ldots+m_{i-1})})$, and $\varphi_{\text{rest}_i,[\alpha]}=(\varphi_{\text{rest}_{i-1},[\alpha]}^{m_i+1}, \ldots, \varphi_{\text{rest}_{i-1},[\alpha]}^{m-(m_1+\ldots+m_{i-1})}) \circ \hat{\Phi}_i$ with the inverse $\hat{\Phi}_i$ of the transformation \eqref{eq:trafo_zeta2_bar}.
        The functions $\varphi_{\text{rest}_i,[-L_{\text{rest}_i}]}$ are again independent of $\zeta_{[-1],\text{rest}_i}$, since otherwise $\operatorname{rank}(\partial_{\zeta_{\text{rest}_{i-1},[-1]}} \varphi_{\text{rest}_{i-1},[-L_i]})$ would have been larger than $m_i$.\\
        
        The above procedure is repeated until, at some step $s$, the condition $m_s=\operatorname{rank}(\partial_{\zeta_{\text{rest}_{s-1},[-1]}} \varphi_{\text{rest}_{s-1},[-L_s]})=\operatorname{dim}(\zeta_{\text{rest}_{s-1},[-1]})$ with the multi-index $L_s=(l_s^1, \ldots, l_s^{m-m_1-\cdots-m_{s-1}})$ defined by
        \begin{equation*}
            \begin{aligned}
            \delta^{-l_s^j}(\varphi_{\text{rest}_{s-1}}^j) & =\varphi_{\text{rest}_{s-1},[-l_s^j]}^j(\zeta_{\text{rest}_{s-1},[-1]}, \ldots, \bar\zeta_{s-1,[-2]}, \\
            &  \hspace{0.5cm} \bar\zeta_{s-1,[-1]}, \ldots, \bar\zeta_{1,[-2]}, \bar\zeta_{1,[-1]}, x) \\
            \delta^{-l_s^j+1}(\varphi_{\text{rest}_{s-1}}^j)& =\varphi_{\text{rest}_{s-1},[-l_s^j+1]}^j(\ldots, \bar\zeta_{s-1,[-2]}, \bar\zeta_{s-1,[-1]},  \\
            & \hspace{0.5cm} \ldots, \bar\zeta_{1,[-2]}, \bar\zeta_{1,[-1]}, x, u)
            \end{aligned}
        \end{equation*}
        is satisfied. Finally, the coordinate transformation
        \begin{equation*}
          \begin{aligned}
            \bar\zeta_{s,[-1]}
            &=
            \varphi_{s,[-\lambda_s]}
            (\zeta_{\text{rest}_{s-1},[-1]},\ldots,\bar\zeta_{s-1,[-2]},\\
            &\hspace{0.5cm}
            \bar\zeta_{s-1,[-1]},\ldots,\bar\zeta_{1,[-2]},\bar\zeta_{1,[-1]},x)
          \end{aligned}
        \end{equation*}
        replaces the last remaining coordinates $\zeta_{\text{rest}_{s-1},[-1]}$ by $\bar\zeta_{s,[-1]}$.
        Analogously to the forward-flat case, the transformations
        \begin{equation*}\label{eq:trafo_zeta_prolongation}
          \begin{aligned}
            \bar\zeta_{1,[-1]}
            &=
            \varphi_{1,[-\lambda_1]}(\zeta_{[-1]},x), \\
            \bar\zeta_{2,[-1]}
            &=
            \varphi_{2,[-\lambda_2]}
            (\zeta_{\text{rest}_{1},[-1]},\ldots,
            \bar\zeta_{1,[-2]},\bar\zeta_{1,[-1]},x), \\
            &\hspace{0.2cm}\vdots \\
            \bar\zeta_{s,[-1]}
            &=
            \varphi_{s,[-\lambda_s]}
            (\zeta_{\text{rest}_{s-1},[-1]},\ldots,\bar\zeta_{s-1,[-2]},\\
            &\hspace{0.5cm}
            \bar\zeta_{s-1,[-1]},\ldots,\bar\zeta_{1,[-2]},\bar\zeta_{1,[-1]},x),
          \end{aligned}
        \end{equation*}
        form an invertible triangular system of equations, which can be abbreviated as $\bar\zeta_{[-1]}=\Phi_\zeta(\zeta_{[-1]},\bar\zeta_{[-R_1+\lambda-1,-1]},x)$.
        The variables $\bar\zeta_{[-1]}$ and their backward-shifts on the right-hand side can be successively substituted, such that the transformation depends only on $x$, $\zeta_{[-1]}$, and backward-shifts of $\zeta_{[-1]}$, and thus again represents a generalized coordinate transformation.

        This procedure results in the special form
        \begin{equation}\label{eq:y_backward_transformed_final}
            \begin{aligned}
            y_{1,[-\lambda_1]} &= \bar\zeta_{1,[-1]} \\
            y_{1,[-\lambda_1+1, 0]} &= \varphi_{1,[-\lambda_1+1, 0]}(x,u) \\
            y_{2,[-\lambda_2]} &= \bar\zeta_{2,[-1]} \\
            y_{2,[-\lambda_2+1, 0]} &= \varphi_{2,[-\lambda_2+1, 0]}(\ldots,\bar\zeta_{1,[-2]},\bar\zeta_{1,[-1]},x,u) \\
            &\hspace{0.2cm}\vdots\\
            y_{s-1,[-\lambda_{s-1}]} &= \bar\zeta_{s-1,[-1]} \\
            y_{s-1,[-\lambda_{s-1}+1, 0]} &= \varphi_{s-1,[-\lambda_{s-1}+1, 0]}(\ldots,\bar\zeta_{s-2,[-2]},\\
            &\hspace{1cm}\bar\zeta_{s-2,[-1]},\ldots,\bar\zeta_{1,[-2]},\bar\zeta_{1,[-1]},x,u) \\
            y_{s,[-\lambda_s]} &= \bar\zeta_{s,[-1]} \\
            y_{s,[-\lambda_s+1, 0]} &= \varphi_{s,[-\lambda_s+1, 0]}(\ldots,\bar\zeta_{s-1,[-2]},\bar\zeta_{s-1,[-1]},\\
            &\hspace{2cm}\ldots,\bar\zeta_{1,[-2]},\bar\zeta_{1,[-1]},x,u) \\
            \end{aligned}
        \end{equation}
        of the flat output and its backward-shifts $y_{[-\lambda,0]}$ with $\varphi_s=\varphi_{\text{rest}_{s-1}}$ and $\lambda_s=L_s$.\\

        The following proposition is the counterpart of Theorem~4.5 in \cite{KolarDiwoldGstottnerSchoeberl:2022} for backward-shifts of the flat output which are of the form \eqref{eq:y_backward_transformed_final}.
        \begin{prop}\label{prop:properties_backward_flat}
            For every backward-flat system \eqref{eq:sysEq} with a backward-flat output \eqref{eq:y_backward} and multi-index $\lambda=(\lambda_1, \ldots, \lambda_s)$ derived with the above procedure, the following applies:
            \begin{itemize}
            \item[(i)] the procedure terminates after $s \leq m$ steps
            \item[(ii)] $\lambda-1 \leq R_1$
            \item[(iii)] $\# (\lambda-1) = n$
            \end{itemize}
        \end{prop}
        \begin{proof}
            (i): In every step $i\geq 1$, the remaining components $\varphi_{\text{rest}_{i-1}}$ of the flat output can be backward-shifted until each component depends explicitly on at least one of the remaining coordinates $\zeta_{\text{rest}_{i-1,[-1]}}$, implying that $\operatorname{rank}(\partial_{\zeta_{\text{rest}_{i-1,[-1]}}} \varphi_{\text{rest}_{i-1},[-L_i]}) \geq 1$, where $\varphi_{\text{rest}_0}=\varphi$ and $\zeta_{\text{rest}_{0,[-1]}}=\zeta_{[-1]}$ for $i=1$. Hence, at least one component of $\zeta_{[-1]}$ can be replaced in each step, and the procedure terminates after at most $\operatorname{dim}(\zeta)=m$ steps.
            
            (ii):
            To prove the inequality $\lambda - 1 \leq R_1$, we consider the subset of equations
            \begin{equation}\label{eq:proof_thm1}
            \begin{aligned}
                y_{i,[-\lambda_i^{j_i}+1,0]}^{j_i} &= \varphi_{i,[-\lambda_i^{j_i}+1,0]}^{j_i}
                (\ldots,\bar\zeta_{i-1,[-2]},\bar\zeta_{i-1,[-1]},\ldots, \\
                & \hspace{0.5cm} \bar\zeta_{1,[-2]},\bar\zeta_{1,[-1]},x,u)
            \end{aligned}
            \end{equation}
            with $j_i = 1,\ldots,m_i$, $i = 1,\ldots,s$ of \eqref{eq:y_backward_transformed_final}.
            Substituting the parameterization \eqref{eq:backward_flat_param} of the state and input by the flat output into \eqref{eq:proof_thm1} and renaming the backward-shifts $\bar\zeta_{[-1-\alpha]}$, $\alpha \geq 0$, according to $\bar\zeta_{[-1-\alpha]} = y_{[-\lambda-\alpha]}$, the equations must be satisfied identically.
            Since the backward-shifts $y_{i,[-\lambda_i^{j_i}+1,0]}^{j_i}$ on the left-hand side of \eqref{eq:proof_thm1} are not contained in the quantities $\ldots,\bar\zeta_{i-1,[-2]},\bar\zeta_{i-1,[-1]},\ldots,\bar\zeta_{1,[-2]},\bar\zeta_{1,[-1]}$ on the right-hand side, all variables $y_{i,[-\lambda_i^{j_i}+1,0]}^{j_i}$ on the left-hand sides must occur in the parameterization of $x$ or $u$, i.e.\ $\{-\lambda_i^{j_i}+1,\ldots,0\} \subseteq \{-r_{1,i}^{j_i},\ldots,0\}$. This is only possible if $-\lambda_i^{j_i}+1 \geq -r_{1,i}^{j_i}$ for all $j_i$ and $i$, i.e. $\lambda - 1 \leq R_1$.

            (iii): To prove the relation $\#(\lambda-1) = n$, recall that from $x=F_x(y_{[-R_1,-1]})$ it follows that
            \begin{equation*}
                \mathrm{d}x \in \operatorname{span}\{\mathrm{d}\varphi_{[-R_1,-1]}\}\,.
            \end{equation*}
            In the coordinates constructed for deriving \eqref{eq:y_backward_transformed_final}, the backward-shifts $y_{[-R_1,-\lambda]}$ correspond to backward-shifts of $\bar\zeta$. Hence, only the differentials $\mathrm{d} y_{[-\lambda+1,-1]}$ of the remaining backward-shifts can yield independent linear combinations which are contained in $\operatorname{span}\{\mathrm{d}x\}$. Due to the functional independence of the backward-shifts of the backward-flat output, there exist exactly $\#(\lambda-1)$ such independent linear combinations. Since \(\dim(x)=n\), it follows that
            \begin{equation*}
                \#(\lambda-1)=n\,.
            \end{equation*}
        \end{proof}

        With item (i) and (ii) from Proposition \ref{prop:properties_backward_flat}, the backward-shifts $y_{[-R_1,0]}$ of \eqref{eq:y_backward_transformed_final} actually take the form
        \begin{equation}\label{eq:y_backward_transformed}
            \begin{aligned}
            y_{1,[-r_{1,1},-\lambda_1]} &= \bar\zeta_{1,[-r_{1,1}+\lambda_1-1,-1]} \\
            y_{1,[-\lambda_1+1, 0]} &= \varphi_{1,[-\lambda_1+1, 0]}(x,u) \\
            &\hspace{0.2cm}\vdots\\
            y_{s-1,[-r_{1,s-1},-\lambda_{s-1}]} &= \bar\zeta_{s-1,[-r_{1,s-1}+\lambda_{s-1}-1,-1]} \\
            y_{s-1,[-\lambda_{s-1}+1, 0]} &= \varphi_{s-1,[-\lambda_{s-1}+1, 0]}(\bar\zeta_{1,[-r_{1,1}+\lambda_1-1,-1]},\\
            &\ldots,\bar\zeta_{s-2,[-r_{1,s-2}+\lambda_{s-2}-1,-1]},x,u) \\
            y_{s,[-r_{1,s}, -\lambda_s]} &= \bar\zeta_{s,[-r_{1,s}+\lambda_{s}-1,-1]} \\
            y_{s,[-\lambda_s+1, 0]} &= \varphi_{s,[-\lambda_s+1, 0]}(\bar\zeta_{1,[-r_{1,1}+\lambda_1-1,-1]},\\
            &\ldots,\bar\zeta_{s-1,[-r_{1,s-1}+\lambda_{s-1}-1,-1]},x,u)\,,
            \end{aligned}
        \end{equation}
       where $r_{1,i}=(r_{1,i}^1,\ldots,r_{1,i}^{m_i})$ denotes the components of the multi-index $R_1$ corresponding to the components $y_i=(y_i^1,\ldots,y_i^{m_i})$ of the flat output for $i=1,\ldots,s$.\\
       Based on the obtained structure of the backward-flat output and its backward-shifts $y_{[-R_1,0]}$, we can now propose a special type of dynamic feedback that linearizes such systems.
       
        \begin{thm}\label{thm:prelongation}
            For every backward-flat system \eqref{eq:sysEq} with a backward-flat output of the form $y=\varphi(x,u)$, there exists an invertible coordinate transformation $\zeta_{[-1]}=\hat{\Phi}_\zeta(x,\bar \zeta_{[-R_1+\lambda-1,-1]})$, such that the $d_1$-fold ($d_1=\# R_1-\#\lambda+m$) prelonged system
    		\begin{equation}\label{eq:prelonged_sys}
    			\begin{aligned}
    				\Sigma^{(d_1)}: \bar \zeta^{+}_{1,[-r_{1,1}+\lambda_1-1]} & = \bar \zeta_{1,[-r_{1,1}+\lambda_1]} \\
    				& \vdots \\
    				\bar \zeta^{+}_{1,[-1]} & = \bar g_1(x,u)\\
                    & \vdots \\
                    \bar \zeta^{+}_{s-1,[-r_{1,s-1}+\lambda_{s-1}-1]} & = \bar \zeta_{s-1,[-r_{1,s-1}+\lambda_{s-1}]} \\
    				& \vdots \\
    				\bar \zeta^{+}_{s-1,[-1]} & = \bar g_{s-1}(\bar\zeta_{1,[-r_{1,1}+\lambda_1-1,-1]},\ldots\\
                    &\hspace{0.5cm},\bar\zeta_{s-1,[-r_{1,s-1}+\lambda_{s-1}-1,-1]},x)\\
                    x^{i,+} & =f^i(x, u)
    			\end{aligned}
    		\end{equation}
    		with the state $(\bar \zeta_{1,[-r_{1,1}+\lambda_1-1,-1]}, \ldots,$ $ \bar \zeta_{s-1,[-r_{1,s-1}+\lambda_{s-1}-1,-1]}, x^1, \ldots, x^n)$ and the original input $u$ is exactly linearizable by static feedback.
    	\end{thm}
        
        \begin{proof}
            First, we prove that \eqref{eq:y_backward_transformed} defines a diffeomorphism. Since the backward-shifts of a flat output up to an arbitrary order are functionally independent, it remains to show that the $\#(R_1-\lambda+1)+n+m$ variables $\bar{\zeta}_{[-R_{1}+\lambda-1,-1]}, x, u$ can be expressed in terms of the $\#R_1+m$ backward-shifts $y_{[-R_{1},0]}$.
            Using the identity $\#(\lambda-1)=n$ from item (iii) in Proposition \ref{prop:properties_backward_flat}, it can be shown with
            \begin{equation*}\label{eq:dim_argument_forward}
        		\#(R_1-\lambda+1)+n+m =\# R_1+m
        	\end{equation*}
            that the dimensions indeed match.        
            Again, the diffeomorphism \eqref{eq:y_backward_transformed} is just the inverse of the parameterizing map of the prelonged system \eqref{eq:prelonged_sys} with $\bar g_1=\varphi_{1,[-\lambda_1+1]}(x,u),\ldots,\bar g_s=\varphi_{s,[-\lambda_s+1]}(\bar\zeta_{1,[-r_{1,1}+\lambda_1-1,-1]},\ldots,\bar\zeta_{s-1,[-r_{1,s-1}+\lambda_{s-1}-1,-1]},x)$. Therefore, the prelonged system is static feedback linearizable, because the corresponding parameterizing map is also a diffeomorphism.
        \end{proof}

        \begin{rem}
            As in the forward-flat case, the map \eqref{eq:y_backward_transformed} can be used to transform the static feedback linearizable prelonged system \eqref{eq:prelonged_sys} into Brunovsky normal form.
        \end{rem}

\section{Examples} \label{sec:four}
        Finally, we illustrate the results of Theorems~\ref{thm:prolongation} and~\ref{thm:prelongation} by means of an academic example with $m=3$ inputs and a practical quadrotor example with $m=4$ inputs, respectively.
        \vspace{0.1cm}
        \begin{exmp}
            Consider the discrete-time system
            \begin{equation}\label{eq:sys_academic}
                \begin{aligned}
                    x^{1,+} &= x^1 + x^4\\
                    x^{2,+} &= x^2 + x^5\\
                    x^{3,+} &= x^3 + x^6\\
                    x^{4,+} &= \tfrac{(x^4 + x^7) x^5 x^8 + (x^6 + x^9) u^1}{x^5 x^8}\\
                    x^{5,+} &= u^1 + x^5\\
                    x^{6,+} &= \tfrac{x^3 (x^{5})^2 + (x^6 + x^8 (x^7 + 1)) x^5 + (x^6 + x^9) u^1}{x^5}\\
                    x^{7,+} &= u^2 + x^7\\
                    x^{8,+} &= \tfrac{(x^{5})^2 x^8 + x^5 (x^{8})^2 + x^6 + x^9}{x^5 x^8}\\
                    x^{9,+} &= \tfrac{1}{x^5 (x^{8})^2} (-x^8 (-x^{10} (x^8)^2 + x^3 x^8 - x^6 - x^9) (x^{5})^2 \\
                    & \hspace{0.4cm} - (x^8 (x^7 + 1) - x^9) (x^{8})^2 x^5 + (x^6 + x^9)^2)\\
                    x^{10,+} &= u^3 + x^{10}
                \end{aligned}
            \end{equation}
            with $n=10$ and $m=3$, obtained by applying an explicit Euler discretization with sampling time $T_s=1$ to the structurally flat triangular form of the academic example in \cite{HartlGstoettnerSchoeberl:2026}. Using the forward-flatness test proposed in \cite{KolarDiwoldSchoberl:2023}, it can be shown that the discrete-time system \eqref{eq:sys_academic} is forward-flat with the $x$-flat output $y=(x^1,x^2,x^3)$.\\
            Employing the stepwise procedure in \cite{KolarDiwoldGstottnerSchoeberl:2022} for the forward-flat system \eqref{eq:sys_academic} results in the coordinate transformations ${\bar u}=\Phi_u(x,{\bar u}_{[0,R_2-\kappa]},u)$ given by
            \begin{equation}\label{eq:sys_academic_trafo}
                \begin{aligned}
                    {\bar u}_1 &= \varphi_{1,[2]}(x^1,x^4,x^5,x^6,x^7,x^8,x^9,u^1) \\
                    {\bar u}_2 &= \varphi_{2,[3]}(x^1,x^2,x^4,x^5,x^6,x^7,x^8,x^9,x^{10},{\bar u}_{1,[0,1]},u^2) \\
                    {\bar u}_3 &= \varphi_{3,[5]}(x,{\bar u}_{1,[0,3]},{\bar u}_{2,[0,1]},u^3) 
                \end{aligned}
            \end{equation}
            with $m_i=1$ for $i=1,\ldots,s$, $s=3$, $\varphi=(\varphi_1,\varphi_2,\varphi_3)=(\varphi^1,\varphi^2,\varphi^3)$ and $\kappa=(\kappa_1,\kappa_2,\kappa_3)=(2,3,5)$ together with the corresponding transformation of its forward-shifts
            \begin{equation*}
                \begin{aligned}
                    {\bar u}_{1,[1]} &= \varphi_{1,[3]}(x^1,x^4,x^5,x^6,x^7,x^8,x^9,x^{10},u^1_{[0,1]},u^2) \\
                    {\bar u}_{1,[2]} &= \varphi_{1,[4]}(x^1,x^4,x^5,x^6,x^7,x^8,x^9,x^{10},u^1_{[0,2]},u^2_{[0,1]},\\
                    & \hspace{1.4cm} u^3) \\
                    {\bar u}_{1,[3]} &= \varphi_{1,[5]}(x^1,x^4,x^5,x^6,x^7,x^8,x^9,x^{10},u^1_{[0,3]},u^2_{[0,2]},\\
                    & \hspace{1.4cm} u^3_{[0,1]}) \\[3ex]
                    {\bar u}_{2,[1]} &= \varphi_{2,[4]}(x^1,x^2,x^4,x^5,x^6,x^7,x^8,x^9,x^{10},{\bar u}_{1,[0,2]},\\
                    & \hspace{1.4cm} u^2_{[0,1]},u^3)\,.
                \end{aligned}
            \end{equation*}
            With this transformation, the forward-shifts $y_{[0,R_2]}$ of the forward-flat output take the form
            \begin{equation*}
              \begin{gathered}
                \begin{aligned}
                  y_{1,[0]} &= x^1 \\
                  y_{1,[1]} &= x^1 + x^4 \\
                  y_{1,[2]} &= {\bar u}_1 \\
                  y_{1,[3]} &= {\bar u}_{1,[1]} \\
                  y_{1,[4]} &= {\bar u}_{1,[2]} \\
                  y_{1,[5]} &= {\bar u}_{1,[3]}
                \end{aligned}
                \hspace{1.25cm}
                \begin{aligned}
                  y_{2,[0]} &= x^2 \\
                  y_{2,[1]} &= x^2 + x^5 \\
                  y_{2,[2]} &= x^2 + 2 x^5 + \tfrac{1}{x^6+x^9} \\
                  & \hspace{0.4cm} ({\bar u}_1-x^1-2 x^4-x^7) x^5 x^8 \\
                  y_{2,[3]} &= {\bar u}_2 \\
                  y_{2,[4]} &= {\bar u}_{2,[1]}
                \end{aligned}
                \\[2ex]
                \begin{aligned}
                  y_{3,[0]} &= x^3 \\
                  y_{3,[1]} &= x^3 + x^6 \\
                  y_{3,[2]} &= ({\bar u}_1 + 1 - x^1 - 2 x^4) x^8 + (x^5 + 1) x^3 + 2 x^6 \\
                  y_{3,[3]} &= \varphi_{3,[3]}(x^1,x^3,x^4,x^5,x^6,x^7,x^8,x^9,{\bar u}_{1,[0,1]}) \\
                  y_{3,[4]} &= \varphi_{3,[4]}(x,{\bar u}_{1,[0,2]},{\bar u}_2) \\
                  y_{3,[5]} &= {\bar u}_3
                \end{aligned}
              \end{gathered}
            \end{equation*}
            with $R_2=(r_{2,1},r_{2,2},r_{2,3})=(5,4,5)$.
            According to Theorem \ref{thm:prolongation}, the system \eqref{eq:sys_academic} together with the invertible, generalized input transformation \eqref{eq:sys_academic_trafo} and a $d_2=\#R_2-\#\kappa=4$-fold prolongation, i.e.
            \begin{equation*}
                \begin{aligned}
                    x^{+} & =f(x, \hat{\Phi}_u(x,{\bar u}_{[0,R_2-\kappa]})) \\
                    {\bar u}^{+}_1 &= {\bar u}_{1,[1]} \\
                    {\bar u}^{+}_{1,[1]} &= {\bar u}_{1,[2]} \\
                    {\bar u}^{+}_{1,[2]} &= {\bar u}_{1,[3]} \\
                    {\bar u}^{+}_2 &= {\bar u}_{2,[1]}
                \end{aligned}
            \end{equation*}
            with the $n=10+4=14$-dimensional state $(x^1,\ldots,x^{10},{\bar u}_1,{\bar u}_{1,[1]},{\bar u}_{1,[2]},{\bar u}_{2})$ and new input ${\bar u}_{1,[3]},{\bar u}_{2,[1]},{\bar u}_{3})$, is static feedback linearizable. This can be verified using e.g. the static feedback linearization test of \cite{Grizzle:1986}.
        \end{exmp}
        \vspace{0.1cm}
        
        \begin{exmp}
            The continuous-time model of a quadrotor has already been studied in the context of flatness-based control; see, e.g., \cite{MellingerKumar:2011,GreeffSchoellig:2022}. In this example, we demonstrate the exact linearization of a corresponding discrete-time quadrotor model, which constitutes an essential step towards the design of discrete-time trajectory tracking controllers. By applying the backward-flatness test proposed in \cite{SchrotshamerKolarSchoeberl:2026}, it can be shown that the discrete-time system
            \begin{equation}\label{eq:sys_quadcopter}
                \begin{aligned}
                    x^{1,+} &= x^1 + x^5 T_s + T_s^2 (g + u^1) \alpha_1, \\
                    x^{2,+} &= x^2 + x^6 T_s + T_s^2 (g + u^1) \alpha_2, \\
                    x^{3,+} &= x^3 + T_s x^7 + T_s^2 u^1, \\
                    x^{4,+} &= x^4 + T_s x^8 + T_s^2 u^2, \\
                    x^{5,+} &= x^5 + T_s (g + u^1) \alpha_1, \\
                    x^{6,+} &= x^6 + T_s (g + u^1) \alpha_2, \\
                    x^{7,+} &= x^7 + T_s u^1, \\
                    x^{8,+} &= x^8 + T_s u^2, \\
                    x^{9,+} &= x^9 + T_s x^{11} + T_s^2 u^3, \\
                    x^{10,+} &= x^{10} + T_s x^{12} + T_s^2 u^4, \\
                    x^{11,+} &= x^{11} + T_s u^3, \\
                    x^{12,+} &= x^{12} + T_s u^4
                \end{aligned}
            \end{equation}
            with 
            \begin{align*}
            \alpha_1=&\sin(T_s^2 u^2 + T_s x^8 + x^4) \tan(T_s^2 u^3 + T_s x^{11} + x^9) \\
            &\sec(T_s^2 u^4 + T_s x^{12} + x^{10}) + \cos(T_s^2 u^2 + T_s x^8 + x^4) \\
            &\tan(T_s^2 u^4 + T_s x^{12} + x^{10}) \\
            \alpha_2=&-\cos(T_s^2 u^2 + T_s x^8 + x^4) \tan(T_s^2 u^3 + T_s x^{11} + x^9) \\
            &\sec(T_s^2 u^4 + T_s x^{12} + x^{10}) + \sin(T_s^2 u^2 + T_s x^8 + x^4) \\
            &\tan(T_s^2 u^4 + T_s x^{12} + x^{10})
            \end{align*}
            and $n=12$ and $m=4$ is backward-flat with the flat output
            \begin{equation}\label{eq:quadcopter_y}
                \begin{aligned}
                y=&(x^1 + x^5 T_s + T_s^2 (g + u^1) \alpha_1,x^2 + x^6 T_s + \\
                & T_s^2 (g + u^1) \alpha_2,x^3 + T_s x^7 + T_s^2 u^1,x^4 + T_s x^8 + T_s^2 u^2)\,.
            \end{aligned}
            \end{equation}
            This model is obtained, for a sampling time $T_s>0$, by first applying a suitable state transformation to the continuous-time model and subsequently discretizing the transformed system using the implicit Euler method. For the functions $g(x)$ of \eqref{eq:sysEq_ext} defining the variables $\zeta$, we choose $g(x)=(x^7,x^8,x^{11},x^{12})$.

            In the first step, the variables $\zeta_{[-1]}$ first occur in the third backward-shift for each component of the backward-flat output \eqref{eq:quadcopter_y}, i.e. $L_1=(3,3,3,3)$. Since $m_1=\operatorname{rank}(\partial_{\zeta_{[-1]}}\varphi_{[-L_1]})=2$, two components of $\zeta_{[-1]}$ can be replaced in this step. For this purpose, we reorder the components of the flat output \eqref{eq:quadcopter_y} and $\zeta$ such that $\operatorname{rank}(\partial_{\zeta_{1,[-1]}}\varphi_{1,[-\lambda_1]})=2$ with $\varphi_1=(\varphi^3,\varphi^4)$, $\zeta_{1}=(\zeta^1,\zeta^2)$ and $\lambda_1=(3,3)$.
            Thus, with the transformation
            \begin{subequations} \label{eq:step_1_trafo}
                \begin{align}
                    \bar\zeta^1_{1,[-1]} &=\varphi_{1,[-3]}^1(\zeta_{1,[-1]}^1,x^3,x^7)=(-x^7 - \zeta^1_{1,[-1]})T_s + x^3  \label{subeq:step_1_trafo1} \\
                    \bar\zeta^2_{1,[-1]} &=\varphi_{1,[-3]}^2(\zeta_{1,[-1]}^2,x^4,x^8)=(-x^8 - \zeta^2_{1,[-1]})T_s + x^4\,, \label{subeq:step_1_trafo2}
                \end{align}
            \end{subequations}
            the backward-shifts of $y_{[-L_1,0]}$ take the form
            \begin{equation*}
              \begin{aligned}
                  y^1_{1,[-3]} &= (-x^7 - \zeta^1_{1,[-1]})T_s + x^3=\bar\zeta^1_{1,[-1]}\\
                  y^1_{1,[-2]} &= x^3 - T_s x^7\\
                  y^1_{1,[-1]} &= x^3\\
                  y^1_{1,[0]} &= x^3 + T_s x^7 + T_s^2 u^1 \\[2ex]
                  y^2_{1,[-3]} &= (-x^8 - \zeta^2_{1,[-1]})T_s + x^4=\bar\zeta^2_{1,[-1]}\\
                  y^2_{1,[-2]} &= x^4 - T_s x^8\\
                  y^2_{1,[-1]} &= x^4\\
                  y^2_{1,[0]} &= x^4 + T_s x^8 + T_s^2 u^2 \\[2ex]
                  y^1_{\text{rest}_1,[-3]} &= \varphi^1_{\text{rest}_1,[-3]}(\bar\zeta^1_{1,[-1]},x^1,x^3,x^4,x^5,x^7,x^9,x^{10})\\
                  y^1_{\text{rest}_1,[-2]} &= x^1 - T_s x^5\\
                  y^1_{\text{rest}_1,[-1]} &= x^1\\
                  y^1_{\text{rest}_1,[0]} &= x^1 + T_s x^5 + T_s^2 (g + u^1) \alpha_1\\[2ex]
                  y^2_{\text{rest}_1,[-3]} &= \varphi^2_{\text{rest}_1,[-3]}(\bar\zeta^1_{1,[-1]},x^2,x^3,x^4,x^6,x^7,x^9,x^{10})\\
                  y^2_{\text{rest}_1,[-2]} &= x^2 - T_s x^6\\
                  y^2_{\text{rest}_1,[-1]} &= x^2\\
                  y^2_{\text{rest}_1,[0]} &= x^2 + T_s x^6 + T_s^2 (g + u^1) \alpha_2
              \end{aligned}
            \end{equation*}
            with $y_1=(y^3, y^4)$, $y_{\text{rest}_1}=(y^1, y^2)$, $L_{\text{rest}_1}=(3, 3)$, $\varphi_{\text{rest}_1,[\alpha]}=(\varphi_{[\alpha]}^{1}, \varphi_{[\alpha]}^2) \circ \hat{\Phi}_1$ and $\hat{\Phi}_1$ denoting the inverse of the transformation \eqref{eq:step_1_trafo}.\\

            Next, the rest of the components of the flat output $y_{\text{rest}_1}$ are shifted further back, where the remaining variables $\zeta_{\text{rest}_1,[-1]}$ now appear in the fifth backward-shift, yielding $L_2=(5,5)$. Due to $m_2=\operatorname{rank}(\partial_{\zeta_{\text{rest}_{1},[-1]}} \varphi_{\text{rest}_{1},[-L_2]})=\operatorname{dim}(\zeta_{\text{rest}_{1},[-1]})=2$, all remaining components of $\zeta_{\text{rest}_1,[-1]}$ can be replaced. Hence, the procedure terminates at the step $s=2$ with $\varphi_2=(\varphi^1_{\text{rest}_{1}},\varphi^2_{\text{rest}_{1}})$, $y_2=(y^1_{\text{rest}_{1}},y^2_{\text{rest}_{1}})$, $\zeta_{2}=(\zeta^1_{\text{rest}_{1}},\zeta^2_{\text{rest}_{1}})$ and $\lambda_2=(5,5)$. Now the transformation
            \begin{equation}\label{eq:step_2_trafo}
            \begin{aligned}
                \bar\zeta^1_{2,[-1]} &=\varphi^1_{2,[-5]}(\zeta_{\text{rest}_{1},[-1]}, \bar\zeta^1_{1,[-3,-1]}, x^1,x^3,x^4,x^5,\\
                &\hspace{0.5cm}x^7,x^8,x^9,x^{10},x^{11},x^{12})\\
                \bar\zeta^2_{2,[-1]} &=\varphi^2_{2,[-5]}(\zeta_{\text{rest}_{1},[-1]}, \bar\zeta^1_{1,[-3,-1]}, x^2,x^3,x^4,x^6\\
                &\hspace{0.5cm},x^7,x^8,x^9,x^{10},x^{11},x^{12})
            \end{aligned}
        \end{equation}
        together with the transformation of the backward-shifts of \eqref{subeq:step_1_trafo1}, that is,
        \begin{equation}\label{eq:step_1_trafo_1}
        \begin{aligned}
            \bar\zeta^1_{1,[-2]} &=\varphi_{1,[-4]}^1(\zeta_{1,[-2,-1]}^1,x^3,x^7)\\
            \bar\zeta^1_{1,[-3]} &=\varphi_{1,[-5]}^1(\zeta_{1,[-3,-1]}^1,x^3,x^7)
        \end{aligned}
        \end{equation}
        results in
            \begin{equation*}
              \begin{gathered}
                  \begin{aligned}
                  y^1_{1,[-3]} &= \bar\zeta^1_{1,[-1]}\\
                  y^1_{1,[-2]} &= x^3 - T_s x^7\\
                  y^1_{1,[-1]} &= x^3\\
                  y^1_{1,[0]} &= x^3 + T_s x^7 + T_s^2 u^1
                \end{aligned}
                \qquad
                \begin{aligned}
                  y^2_{1,[-3]} &= \bar\zeta^2_{1,[-1]}\\
                  y^2_{1,[-2]} &= x^4 - T_s x^8\\
                  y^2_{1,[-1]} &= x^4\\
                  y^2_{1,[0]} &= x^4 + T_s x^8 + T_s^2 u^2
                \end{aligned}
              \end{gathered}
            \end{equation*}

            \begin{equation*}
              \begin{gathered}
                \begin{aligned}
                  y^1_{2,[-5]} &= \bar\zeta^1_{2,[-1]}\\
                  y^1_{2,[-4]} &= \varphi^1_{2,[-4]}(\bar\zeta^1_{1,[-2,-1]},x^1,x^3,x^4,x^5,x^7,x^8,x^9,\\
                  &\hspace{0.5cm}x^{10},x^{11},x^{12})\\
                  y^1_{2,[-3]} &= \varphi^1_{2,[-3]}(\bar\zeta^1_{1,[-1]},x^1,x^3,x^4,x^5,x^7,x^9,x^{10})\\
                  y^1_{2,[-2]} &= x^1 - T_s x^5\\
                  y^1_{2,[-1]} &= x^1\\
                  y^1_{2,[0]} &= x^1 + T_s x^5 + T_s^2 (g + u^1) \alpha_1\\[2ex]
                  y^2_{2,[-5]} &= \bar\zeta^2_{2,[-1]}\\
                  y^2_{2,[-4]} &= \varphi^2_{2,[-4]}(\bar\zeta^1_{1,[-2,-1]},x^2,x^3,x^4,x^6,x^7,x^8,x^9,\\
                  &\hspace{0.5cm}x^{10},x^{11},x^{12})\\
                  y^2_{2,[-3]} &= \varphi^2_{2,[-3]}(\bar\zeta^1_{1,[-1]},x^2,x^3,x^4,x^6,x^7,x^9,x^{10})\\
                  y^2_{2,[-2]} &= x^2 - T_s x^6\\
                  y^2_{2,[-1]} &= x^2\\
                  y^2_{2,[0]} &= x^2 + T_s x^6 + T_s^2 (g + u^1) \alpha_2\,,
              \end{aligned}
              \end{gathered}
            \end{equation*}
            which is precisely of the form \eqref{eq:y_backward_transformed_final}.
            It is now straightforward to verify that the parameterization of the system variables $x$ and $u$ requires the backward-shifts of $y$ up to the orders $R_1=(4,2,4,4)$. Hence, according to \eqref{eq:y_backward_transformed}, the backward-shifts $y_{[-R_1,0]}$ are given by
            \begin{equation*}
              \begin{gathered}
                  \begin{aligned}
                  y^1_{1,[-4]} &= \bar\zeta^1_{1,[-2]}\\
                  y^1_{1,[-3]} &= \bar\zeta^1_{1,[-1]}\\
                  y^1_{1,[-2]} &= x^3 - T_s x^7\\
                  y^1_{1,[-1]} &= x^3\\
                  y^1_{1,[0]} &= x^3 + T_s x^7 + T_s^2 u^1
                \end{aligned}
                \hspace{0.5cm}
                \begin{aligned}
                  y^2_{1,[-2]} &= x^4 - T_s x^8\\
                  y^2_{1,[-1]} &= x^4\\
                  y^2_{1,[0]} &= x^4 + T_s x^8 + T_s^2 u^2
                \end{aligned}
                \\[2ex]
              \begin{aligned}
                  y^1_{2,[-4]} &= \varphi^1_{2,[-4]}(\bar\zeta^1_{1,[-2,-1]},x^1,x^3,x^4,x^5,x^7,x^8,x^9,\\
                  &\hspace{0.5cm}x^{10},x^{11},x^{12})\\
                  y^1_{2,[-3]} &= \varphi^1_{2,[-3]}(\bar\zeta^1_{1,[-1]},x^1,x^3,x^4,x^5,x^7,x^9,x^{10})\\
                  y^1_{2,[-2]} &= x^1 - T_s x^5\\
                  y^1_{2,[-1]} &= x^1\\
                  y^1_{2,[0]} &= x^1 + T_s x^5 + T_s^2 (g + u^1) \alpha_1\\[2ex]
                  y^2_{2,[-4]} &= \varphi^2_{2,[-4]}(\bar\zeta^1_{1,[-2,-1]},x^2,x^3,x^4,x^6,x^7,x^8,x^9,\\
                  &\hspace{0.5cm}x^{10},x^{11},x^{12})\\
                  y^2_{2,[-3]} &= \varphi^2_{2,[-3]}(\bar\zeta^1_{1,[-1]},x^2,x^3,x^4,x^6,x^7,x^9,x^{10})\\
                  y^2_{2,[-2]} &= x^2 - T_s x^6\\
                  y^2_{2,[-1]} &= x^2\\
                  y^2_{2,[0]} &= x^2 + T_s x^6 + T_s^2 (g + u^1) \alpha_2\,.
              \end{aligned}
              \end{gathered}
            \end{equation*}
            According to Theorem \ref{thm:prelongation}, the system \eqref{eq:sys_quadcopter} together with the invertible, coordinate transformations \eqref{eq:step_1_trafo}, \eqref{eq:step_2_trafo}, \eqref{eq:step_1_trafo_1} and a $d_1=\#R_1-\#\lambda+m=2$-fold prelongation, i.e.
            \begin{equation*}
                \begin{aligned}
                    \bar\zeta^{1,+}_{1,[-2]} &= \bar\zeta^{1}_{1,[-1]} \\
                    \bar\zeta^{1,+}_{1,[-1]} &= \bar g^1_{1}(x^3,x^7) \\
                    x^{+} & =f(x, u)
                \end{aligned}
            \end{equation*}
            with the $n=12+2=14$-dimensional state $(\bar\zeta^{1}_{1,[-2]},\bar\zeta^{1}_{1,[-1]},x^1,\ldots,x^{12})$, original input $u$ and $\bar g^1_1=x^3-T_s x^7$, is static feedback linearizable. This can once again be verified using e.g. the static feedback linearization test of \cite{Grizzle:1986}.
        \end{exmp}



\bibliographystyle{IEEEtran}

\bibliography{IEEEabrv,bibliography_Johannes}             

                 
\end{document}